\documentclass[reqno,12pt]{amsart}

\usepackage{amsmath, epsfig, cite}
\usepackage{amssymb}
\usepackage{amsfonts}
\usepackage{latexsym}
\usepackage{amsthm}

\newtheorem{theorem}{Theorem}

\newtheorem{conjecture}{Conjecture}
\newtheorem{lemma}{Lemma}

\theoremstyle{remark}

\numberwithin{equation}{section}

\allowdisplaybreaks[1]

\author{Victor J.\ W.\ Guo}
\address{School of Mathematics and Statistics, Huaiyin Normal University,
Huai'an 223300, Jiangsu, People's Republic of China}
\email{jwguo@hytc.edu.cn}
\thanks{The first author was partially supported by the National Natural
Science Foundation of China (grant 11771175).}

\author{Michael J.\ Schlosser}
\address{Fakult\"at f\"ur Mathematik, Universit\"at Wien,
Oskar-Morgenstern-Platz~1, A-1090 Vienna, Austria}
\email{michael.schlosser@univie.ac.at}
\thanks{The second author was partially supported by FWF Austrian Science Fund
grant P 32305.}

\title[A family of $q$-hypergeometric congruences modulo $\Phi_n(q)^4$]
{A family of $q$-hypergeometric congruences modulo\\[2pt]the fourth power
of a cyclotomic polynomial}

\subjclass[2010]{Primary 33D15; Secondary 11A07, 11B65}

\keywords{basic hypergeometric series; supercongruences; $q$-congruences;
cyclotomic polynomial; Andrews' transformation, Gasper's summation.}

\begin{document}

\begin{abstract}
We prove a two-parameter family of $q$-hypergeometric congruences modulo
the fourth power of a cyclotomic polynomial. Crucial ingredients in our proof
are George Andrews'
multiseries extension of the Watson transformation, and a Karlsson--Minton type summation for very-well-poised basic hypergeometric series due to George Gasper.
The new family of $q$-congruences is then used to prove two conjectures
posed earlier by the authors.
\end{abstract}

\maketitle

\section{Introduction}
In 1914, Ramanujan \cite{Ramanujan} presented a number of fast
approximations of $1/\pi$, including
\begin{align}
\sum_{k=0}^{\infty}(6k+1)\frac{(\frac{1}{2})_k^3}{k!^3 4^k}
=\frac{4}{\pi}, \label{eq:ram}
\end{align}
where $(a)_n=a(a+1)\cdots(a+n-1)$ denotes the rising factorial.
In 1997, Van Hamme \cite{Hamme} proposed 13 interesting $p$-adic
analogues of Ramanujan-type formulas, such as
\begin{align}
\sum_{k=0}^{(p-1)/2}(6k+1)\frac{(\frac{1}{2})_k^3}{k!^3 4^k}
&\equiv p(-1)^{(p-1)/2}\pmod{p^4}, \label{eq:pram}
\end{align}
where  $p>3$ is a prime. Van Hamme's supercongruence
\eqref{eq:pram} was first proved by Long \cite{Long}.
It should be pointed out that all of the 13 supercongruences have
been proved by different techniques (see \cite{OZ,Swisher}). For
some background on Ramanujan-type supercongruences,
the reader is referred to Zudilin's paper \cite{Zud2009}.

In 2016, Long and Ramakrishna~\cite[Thm.~2]{LR} proved the following
supercongruence:
\begin{equation}
\sum_{k=0}^{p-1} (6k+1) \frac{(\frac{1}{3})_k^6}{k!^6} \equiv
\begin{cases} -p\displaystyle \Gamma_p\bigg(\frac{1}{3}\bigg)^9
\pmod{p^6}, &\text{if $p\equiv 1\pmod 6$,}\\[10pt]
 -\frac{p^4}{27}\displaystyle \Gamma_p\bigg(\frac{1}{3}\bigg)^9\pmod{p^6},
&\text{if $p\equiv 5\pmod 6$,}
\end{cases} \label{eq:d2}
\end{equation}
where $\Gamma_p(x)$ is the $p$-adic Gamma function.
This result for $p\equiv 1\pmod 6$ confirms the (D.2) supercongruence
of Van Hamme, which asserts a congruence modulo $p^4$.

During the past few years, many congruences and supercongruences
were generalized to the $q$-setting by various authors (see, for instance,
\cite{Gorodetsky,Guo2018,Guo2,Guo-gz,Guo-fac,Guo-m3,Guo-par,
GJZ,GS19,GSdiff,GS0,GS,GW2,GuoZu,NP,NP2,Straub,Tauraso2}).
In particular, the authors \cite[Thm.~2.3]{GS} proposed the following
partial $q$-analogue of Long and Rama\-krishna's
supercongruence \eqref{eq:d2}:
\begin{equation}
\sum_{k=0}^{n-1} [6k+1]\frac{(q;q^3)_k^6}{(q^3;q^3)_k^6} q^{3k}
\equiv
\begin{cases} 0  \pmod{[n]}, &\text{if $n\equiv 1\pmod 3$,}\\[10pt]
 0 \pmod{[n]\Phi_n(q)}, &\text{if $n\equiv 2\pmod 3$.}
\end{cases}  \label{eq;3rd-noa}
\end{equation}
Here and throughout the paper, we adopt the standard $q$-notation
(cf.\ \cite{GR}):
For an indeterminate $q$, let
\begin{equation*}
(a;q)_n=(1-a)(1-aq)\cdots (1-aq^{n-1})
\end{equation*}
be the {\em $q$-shifted factorial}.
For convenience, we compactly write
\begin{equation*}
(a_1,a_2,\ldots,a_m;q)_n=(a_1;q)_n (a_2;q)_n\cdots (a_m;q)_n.
\end{equation*}
Moreover,
$
[n]=[n]_q=1+q+\cdots+q^{n-1}
$
denotes the {\em $q$-integer}
and $\Phi_n(q)$ the $n$-th {\em cyclotomic polynomial} in $q$,
which may be defined as
\begin{align*}
\Phi_n(q)=\prod_{\substack{1\leqslant k\leqslant n\\ \gcd(n,k)=1}}(q-\zeta^k),
\end{align*}
where $\zeta$ is an $n$-th primitive root of unity.

The authors \cite[Conjectures 12.10 and 12.11]{GS} also proposed the
following conjectures, the first one generalizing
the $q$-congruence \eqref{eq;3rd-noa} for $n\equiv 2\pmod 3$.
\begin{conjecture}\label{conj-1} Let $d\geqslant 3$ and $n$ be positive
integers with $n\equiv -1\pmod{d}$. Then
\begin{equation*}
\sum_{k=0}^{M}[2dk+1]
\frac{(q;q^d)_k^{2d}}{(q^d;q^d)_k^{2d}}q^{d(d-2)k} \equiv 0 \pmod{[n]\Phi_{n}(q)^3},
\end{equation*}
where $M=((d-1)n-1)/d$ or $n-1$.
\end{conjecture}

\begin{conjecture}\label{conj-2} Let $d\geqslant 3$ and $n>1$ be integers
with $n\equiv 1\pmod{d}$. Then
\begin{equation*}
\sum_{k=0}^{M}[2dk-1]
\frac{(q^{-1};q^d)_k^{2d}}{(q^d;q^d)_k^{2d}}q^{d^2 k} \equiv 0 \pmod{[n]\Phi_{n}(q)^3},
\end{equation*}
where $M=((d-1)n+1)/d$ or $n-1$.
\end{conjecture}

Note that Conjecture \ref{conj-1} does not hold for $d=2$ while
Conjecture \ref{conj-2} is still true for $d=2$.
In fact, the first author and Wang \cite{GW2} proved that
\begin{equation*}
\sum_{k=0}^{(n-1)/2}[4k+1]\frac{(q;q^2)_k^4}{(q^2;q^2)_k^4}
\equiv q^{(1-n)/2}[n]+\frac{(n^2-1)(1-q)^2}{24}q^{(1-n)/2}[n]^3
\pmod{[n]\Phi_n(q)^3}
\end{equation*}
for odd $n$, and the authors \cite{GSdiff} showed that
\begin{equation*}
\sum_{k=0}^{(n+1)/2}[4k-1]\frac{(q^{-1};q^2)_k^4}{(q^2;q^2)_k^4} q^{4k}
\equiv -(1+3q+q^2)[n]^4 \pmod{[n]^4\Phi_n(q)}
\end{equation*}
for odd $n>1$.

The last two $q$-congruences are quite special, as they are rare examples of
$q$-hyper\-geometric congruences that were rigorously shown
to hold modulo a high (fourth and even fifth) power
of a cyclotomic polynomial.
The main purpose of this paper is to add a complete two-parameter family of
$q$-hypergeometric congruences to the list of such $q$-congruences (see Theorem~\ref{thm:1}).

We shall also prove that Conjectures~\ref{conj-1} and \ref{conj-2} are true.
Our proof relies on the following result:
\begin{theorem}\label{thm:1}
Let $d,r,n$ be integers satisfying $d\geqslant 3$, $r\leqslant d-2$
(in particular, $r$ may be negative), and $n\geqslant d-r$, such that
$d$ and $r$ are coprime, and $n\equiv-r\pmod{d}$.
Then
\begin{equation}
\sum_{k=0}^{n-1}[2dk+r]
\frac{(q^r;q^d)_k^{2d}}{(q^d;q^d)_k^{2d}}q^{d(d-1-r)k} \equiv 0 \pmod{\Phi_{n}(q)^4}.
\label{eq:thm1}
\end{equation}
\end{theorem}

This result is similar in nature to the two-parameter result in \cite[Thm.~1.1]{Guo-fac}
which, however, only concerned a $q$-congruence modulo $\Phi_n(q)^2$.

Note that the $q$-congruence \eqref{eq:thm1} is still true when the sum
is over $k$ from $0$ to $((d-1)n-r)/d$, since
$(q^r;q^d)_k/(q^d;q^d)_k \equiv 0\pmod{\Phi_n(q)}$ for
$((d-1)n-r)/d<k\leqslant n-1$.
(Also, we must have $((d-1)n-r)/d\le n-1$ since $n\geqslant d-r$.)
Thus, Theorem \ref{thm:1} implies that Conjectures~\ref{conj-1} and \ref{conj-2}
hold modulo $\Phi_n(q)^4$.

To prove that Conjectures~\ref{conj-1} and \ref{conj-2} also hold modulo $[n]$
(which in conjunction with Theorem~\ref{thm:1} would fully establish the
validity of the conjectures),
we need to prove the following result.
\begin{theorem}\label{thm:2} Let $d\geqslant 3$ and $n$ be positive
integers with $\gcd(d,n)=1$. Then
\begin{align}
\sum_{k=0}^{n-1}[2dk+1]
\frac{(q;q^d)_k^{2d}}{(q^d;q^d)_k^{2d}}q^{d(d-2)k} &\equiv 0 \pmod{\Phi_n(q)}, \label{eq:first-1}\\
\intertext{and}
\sum_{k=0}^{n-1}[2dk-1]
\frac{(q^{-1};q^d)_k^{2d}}{(q^d;q^d)_k^{2d}}q^{d^2 k} &\equiv 0 \pmod{\Phi_n(q)}. \label{eq:first-2}
\end{align}
\end{theorem}

We shall prove Theorem~\ref{thm:1}
in Section~\ref{sec:thm1}
by making a careful use of Andrews' multiseries
generalization \eqref{andrews} of the Watson
transformation~\cite[Theorem~4]{Andrews75},
combined with a special case of Gasper's very-well-poised
Karlsson--Minton type summation \cite[Eq.~(5.13)]{Gasper}.
We point out that Andrews' transformation
plays an important role in combinatorics and number theory. For example,
this transformation was utilized by Zudilin \cite{Zu} to solve a problem
of Asmus Schmidt. It was also used by Krattenthaler and Rivoal \cite{KR}
to provide an
alternative proof of a result by Zudilin that relates a very-well-poised
hypergeometric series with a Vasilenko--Vasilev-type multiple integral,
the latter serving as a tool in the study of the arithmetic behaviour
of values of the Riemann zeta function at integers.
Andrews' transformation was also used by the first author,
Jouhet and Zeng \cite{GJZ} to prove some $q$-congruences involving
$q$-binomial coefficients.
The couple Hessami Pilehrood \cite{HH} used this transformation
to give a short proof of a theorem of Zagier.
Recently, the present authors \cite{GS19,GS0} applied Andrews' transformation
to establish some $q$-congruences for truncated basic hypergeometric series.
We shall prove Theorem \ref{thm:2} in Section~\ref{sec:thm2}.
The proof of Conjectures~\ref{conj-1} and \ref{conj-2} will be given in Section~\ref{sec:conjectures}.
We conclude this short paper by Section~\ref{sec:final}, where we state
an open problem involving a $q$-hypergeometric congruence modulo the
fifth power of a cyclotomic polynomial.

\section{Proof of Theorem \ref{thm:1}}\label{sec:thm1}
We first give a simple $q$-congruence modulo $\Phi_n(q)^2$, which was
already used in \cite{GS0}.
\begin{lemma}\label{lem:mod-square}
Let $\alpha$, $r$ be integers and $n$ a positive integer. Then
\begin{equation}
(q^{r-\alpha  n},q^{r+\alpha  n};q^d)_k \equiv (q^r;q^d)_k^2 \pmod{\Phi_n(q)^2}.
\label{eq:mod-square}
\end{equation}
\end{lemma}
\begin{proof}
For any integer $j$, it is easy to check that
\begin{equation*}
(1-q^{\alpha n-dj+d-r})(1-q^{\alpha  n+dj-d+r})+(1-q^{dj-d+r})^2
q^{\alpha  n-dj+d-r}=(1-q^{\alpha  n})^2
\end{equation*}
and $1-q^{\alpha n}\equiv 0\pmod{\Phi_n(q)}$, and so
\begin{equation*}
(1-q^{\alpha  n-dj+d-r})(1-q^{\alpha  n+dj-d+r})\equiv -(1-q^{dj-d+r})^2
q^{\alpha  n-dj+d-r}\pmod{\Phi_n(q)^2}.
\end{equation*}
The proof then follows easily from the above $q$-congruence.
\end{proof}

We will make use of a powerful transformation formula due to
Andrews \cite[Theorem~4]{Andrews75}, which can be stated as follows:
\begin{align}
\sum_{k\geqslant 0}\frac{(a,q\sqrt{a},-q\sqrt{a},b_1,c_1,\dots,b_m,c_m,q^{-N};q)_k}
{(q,\sqrt{a},-\sqrt{a},aq/b_1,aq/c_1,\dots,aq/b_m,aq/c_m,aq^{N+1};q)_k}
\left(\frac{a^mq^{m+N}}{b_1c_1\cdots b_mc_m}\right)^k &\notag\\[5pt]
=\frac{(aq,aq/b_mc_m;q)_N}{(aq/b_m,aq/c_m;q)_N}
\sum_{j_1,\dots,j_{m-1}\geqslant 0}
\frac{(aq/b_1c_1;q)_{j_1}\cdots(aq/b_{m-1}c_{m-1};q)_{j_{m-1}}}
{(q;q)_{j_1}\cdots(q;q)_{j_{m-1}}} \notag\\[5pt]
\times\frac{(b_2,c_2;q)_{j_1}\dots(b_m,c_m;q)_{j_1+\dots+j_{m-1}}}
{(aq/b_1,aq/c_1;q)_{j_1}
\dots(aq/b_{m-1},aq/c_{m-1};q)_{j_1+\dots+j_{m-1}}}& \notag\\[5pt]
\times\frac{(q^{-N};q)_{j_1+\dots+j_{m-1}}}
{(b_mc_mq^{-N}/a;q)_{j_1+\dots+j_{m-1}}}
\frac{(aq)^{j_{m-2}+\dots+(m-2)j_1} q^{j_1+\dots+j_{m-1}}}
{(b_2c_2)^{j_1}\cdots(b_{m-1}c_{m-1})^{j_1+\dots+j_{m-2}}}&.  \label{andrews}
\end{align}
This transformation actually constitutes a multiseries generalization
of Watson's
$_8\phi_7$ transformation formula (see \cite[Appendix (III.18)]{GR})
which we state here in standard notation for basic hypergeometric series
\cite[Section 1]{GR}:
\begin{align}
& _{8}\phi_{7}\!\left[\begin{array}{cccccccc}
a,& qa^{\frac{1}{2}},& -qa^{\frac{1}{2}}, & b,    & c,    & d,    & e,    & q^{-n} \\
  & a^{\frac{1}{2}}, & -a^{\frac{1}{2}},  & aq/b, & aq/c, & aq/d, & aq/e, & aq^{n+1}
\end{array};q,\, \frac{a^2q^{n+2}}{bcde}
\right] \notag\\[5pt]
&\quad =\frac{(aq, aq/de;q)_n}
{(aq/d, aq/e;q)_n}
\,{}_{4}\phi_{3}\!\left[\begin{array}{c}
aq/bc,\ d,\ e,\ q^{-n} \\
aq/b,\, aq/c,\, deq^{-n}/a
\end{array};q,\, q
\right].  \label{eq:8phi7}
\end{align}

Next, we recall the following very-well-poised Karlsson--Minton type summation
by Gasper~\cite[Eq.~(5.13)]{Gasper} (see also \cite[Ex.~2.33 (i)]{GR}).
\begin{align}
\sum_{k=0}^\infty\frac{(a,q\sqrt{a},-q\sqrt{a},b,a/b,d,e_1,aq^{n_1+1}/e_1,\dots,
e_m,aq^{n_m+1}/e_m;q)_k}{(q,\sqrt{a},-\sqrt{a},aq/b,bq,aq/d,aq/e_1,e_1q^{-n_1},
\dots,aq/e_m,e_mq^{-n_m};q)_k}\left(\frac{q^{1-\nu}}d\right)^k&\notag\\
=\frac{(q,aq,aq/bd,bq/d;q)_\infty}{(bq,aq/b,aq/d,q/d;q)_\infty}
\prod_{j=1}^m\frac{(aq/be_j,bq/e_j;q)_{n_j}}{(aq/e_j,q/e_j;q)_{n_j}}&,
\label{eq:gasper}
\end{align}
where $n_1,\dots,n_m$ are nonnegative integers,
$\nu=n_1+\cdots+n_m$, and $|q^{1-\nu}/d|<1$
when the series does not terminate.
For an elliptic extension of the terminating $d=q^{-\nu}$ case of
\eqref{eq:gasper}, see \cite[Eq.~(1.7)]{RS}.

In particular, we note that for $d=bq$ the right-hand side of
\eqref{eq:gasper} vanishes.
Putting in addition $b=q^{-N}$ we obtain the following terminating summation:
\begin{equation}\label{eq:vwp-km}
\sum_{k=0}^N\frac{(a,q\sqrt{a},-q\sqrt{a},e_1,aq^{n_1+1}/e_1,\dots,
e_m,aq^{n_m+1}/e_m,q^{-N};q)_k}{(q,\sqrt{a},-\sqrt{a},aq/e_1,e_1q^{-n_1},
\dots,aq/e_m,e_mq^{-n_m},aq^{N+1};q)_k}q^{(N-\nu)k}=0,
\end{equation}
valid for $N>\nu=n_1+\cdots+n_m$.

By suitably combining \eqref{andrews} with \eqref{eq:vwp-km},
we obtain the following multiseries summation formula:
\begin{lemma}\label{lem:ms=0}
Let $m\geqslant 2$. Let $q$, $a$ and $e_1,\dots,e_{m+1}$ be arbitrary
parameters with
$e_{m+1}=e_1$, and let $n_1,\dots,n_m$ and $N$ be nonnegative integers
such that $N>n_1+\cdots+n_m$. Then
\begin{align}
0=\sum_{j_1,\dots,j_{m-1}\geqslant 0}
\frac{(e_1q^{-n_1}/e_2;q)_{j_1}\cdots(e_{m-1}q^{-n_{m-1}}/e_m;q)_{j_{m-1}}}
{(q;q)_{j_1}\cdots(q;q)_{j_{m-1}}} \notag\\[5pt]
\times\frac{(aq^{n_2+1}/e_2,e_3;q)_{j_1}\dots
(aq^{n_m+1}/e_m,e_{m+1};q)_{j_1+\dots+j_{m-1}}}
{(e_1q^{-n_1},aq/e_2;q)_{j_1}
\dots(e_{m-1}q^{-n_{m-1}},aq/e_m;q)_{j_1+\dots+j_{m-1}}}& \notag\\[5pt]
\times\frac{(q^{-N};q)_{j_1+\dots+j_{m-1}}}
{(e_1q^{n_m-N+1}/e_m;q)_{j_1+\dots+j_{m-1}}}
\frac{(aq)^{j_{m-2}+\dots+(m-2)j_1} q^{j_1+\dots+j_{m-1}}}
{(aq^{n_2+1}e_3/e_2)^{j_1}\cdots
(aq^{n_{m-1}+1}e_m/e_{m-1})^{j_1+\dots+j_{m-2}}}&. \label{mkm0}
\end{align}
\end{lemma}
\begin{proof}
By specializing the parameters in the multisum transformation \eqref{andrews}
by $b_i\mapsto aq^{n_i+1}/e_i$, $c_i\mapsto e_{i+1}$, for $1\le i\le m$
(where $e_{m+1}=e_1$), and dividing both sides of
the identity by the prefactor of the multisum,
we obtain that the series on the right-hand side of \eqref{mkm0} equals
\begin{align*}
&\frac{(e_mq^{-n_m},aq/e_1;q)_N}{(aq,e_mq^{-n_m}/e_1;q)_N}\\&\times
\sum_{k=0}^N\frac{(a,q\sqrt{a},-q\sqrt{a},e_1,aq^{n_1+1}/e_1,\dots,
e_m,aq^{n_m+1}/e_m,q^{-N};q)_k}{(q,\sqrt{a},-\sqrt{a},aq/e_1,e_1q^{-n_1},
\dots,aq/e_m,e_mq^{-n_m},aq^{N+1};q)_k}q^{(N-\nu)k},
\end{align*}
with $\nu=n_1+\cdots+n_m$.
Now the last sum vanishes by the special case of Gasper's summation stated in
\eqref{eq:vwp-km}.
\end{proof}

We collected enough ingredients and are ready to prove Theorem~\ref{thm:1}.

\begin{proof}[Proof of Theorem \ref{thm:1}]
The left-hand side of \eqref{eq:thm1} can be written as the following
multiple of a terminating $_{2d+4}\phi_{2d+3}$ series:
\begin{align*}
\frac{1-q^r}{1-q}
\sum_{k=0}^{((d-1)n-r)/d}\frac{(q^r,q^{d+\frac{r}{2}},-q^{d+\frac{r}{2}},
\overbrace{q^r,\ldots,q^r}^{\text{$(2d-1)$ times}},q^{d+(d-1)n},q^{r-(d-1)n};q^d)_k}
{(q^d,q^{\frac{r}{2}},-q^{\frac{r}{2}},q^d,\ldots,q^d,q^{r-(d-1)n},q^{d+(d-1)n};q^d)_k}
q^{d(d-1-r)k}.
\end{align*}
Now, by the $m=d$ case of Andrews' transformation \eqref{andrews},
we can write the above expression as
\begin{align}
\frac{(1-q^r)(q^{d+r},q^{-(d-1)n};q^d)_{((d-1)n-r)/d}}
{(1-q)(q^d,q^{r-(d-1)n};q^d)_{((d-1)n-r)/d}}
\sum_{j_1,\dots,j_{d-1}\geqslant 0}
\frac{(q^{d-r};q^d)_{j_1}\cdots(q^{d-r};q^d)_{j_{d-1}}}
{(q^d;q^d)_{j_1}\cdots(q^d;q^d)_{j_{d-1}}}& \notag\\[5pt]
\times\frac{(q^r,q^r;q^d)_{j_1}\dots (q^r,q^r;q^d)_{j_1+\dots+j_{d-2}}
 (q^r,q^{d+(d-1)n};q^d)_{j_1+\dots+j_{d-1}}}
{(q^d,q^d;q^d)_{j_1}
\dots(q^d,q^d;q^d)_{j_1+\dots+j_{d-1}}}& \notag\\[5pt]
\times\frac{(q^{r-(d-1)n};q^d)_{j_1+\dots+j_{d-1}}}
{(q^{d+r};q^d)_{j_1+\dots+j_{d-1}}}
\frac{q^{(d+r)(j_{d-2}+\dots+(d-2)j_1)} q^{d(j_1+\dots+j_{d-1})}}
{q^{2rj_1}\cdots q^{2r(j_1+\dots+j_{d-2})}}&. \label{eq:multi}
\end{align}

It is easy to see that the $q$-shifted factorial $(q^{d+r};q^d)_{((d-1)n-r)/d}$
contains the factor $1-q^{(d-1)n}$ which is a multiple of $1-q^n$.
Similarly, the $q$-shifted factorial $(q^{-(d-1)n};q^d)_{((d-1)n-r)/d}$
contains the factor $1-q^{-(d-1)n}$ (again being a multiple of $1-q^n$)
since $((d-1)n-r)/d\geqslant 1$ holds due to the conditions
$d\geqslant 3$, $r\leqslant d-2$, and $n\geqslant d-r$.
This means that the $q$-factorial $(q^{d+r},q^{-(d-1)n};q^d)_{((d-1)n-r)/d}$
in the numerator of the fraction before the multisummation is divisible by $\Phi_n(q)^2$.
Moreover, it is easily seen that the $q$-factorial $(q^d,q^{r-(d-1)n};q^d)_{((d-1)n-r)/d}$
in the denominator is coprime with $\Phi_n(q)$.

Note that the non-zero terms in the multisummation in \eqref{eq:multi}
are those indexed by $(j_1,\ldots,j_{d-1})$ that satisfy
$j_1+\dots+j_{d-1}\leqslant ((d-1)n-r)/d$ because of the appearance of the
factor $(q^{r-(d-1)n};q^d)_{j_1+\dots+j_{d-1}}$ in the numerator.
None of the factors appearing in the denominator
of the multisummation of \eqref{eq:multi} contain a
factor of the form $1-q^{\alpha n}$ (and are therefore coprime with $\Phi_n(q)$),
except for $(q^{d+r};q^d)_{j_1+\dots+j_{d-1}}$
when $j_1+\dots+j_{d-1}=((d-1)n-r)/d$.
(In this case, the factor $1-q^{(d-1)n}$ appears
in the numerator.) Writing $n=ad-r$ (with $a\geqslant 1$),
we have $j_1+\dots+j_{d-1}=a(d-1)-r$.
Since $r\le d-2$, there must be an $i$ with $j_i\geqslant a$.
Then $(q^{d-r};q^d)_{j_i}$ contains the factor $1-q^{d-r+d(a-1)}=1-q^n$
which is a multiple of $\Phi_n(q)$.
So the denominator of the reduced form of the multisum in
\eqref{eq:multi} is coprime with $\Phi_n(q)$.
What remains is to show that the multisum in \eqref{eq:multi},
without the prefactor, is divisible by $\Phi_n(q)^2$, i.e. vanishes
modulo $\Phi_n(q)^2$.

By repeated application of Lemma~\ref{lem:mod-square},
the mulitsum in \eqref{eq:multi},
without the prefactor, is modulo $\Phi_n(q)^2$ congruent to
\begin{align*}
\sum_{j_1,\dots,j_{d-1}\geqslant 0}
\frac{(q^{d-r};q^d)_{j_1}\cdots(q^{d-r};q^d)_{j_{d-1}}}
{(q^d;q^d)_{j_1}\cdots(q^d;q^d)_{j_{d-1}}}& \notag\\[5pt]
\times\frac{(q^{r-(d-2)n},q^{r+(d-2)n};q^d)_{j_1}\dots
(q^{r-n},q^{r+n};q^d)_{j_1+\dots+j_{d-2}}
 (q^r,q^{d+(d-1)n};q^d)_{j_1+\dots+j_{d-1}}}
{(q^{d+(d-1)n},q^{d-(d-1)n};q^d)_{j_1}
\dots(q^{d+2n},q^{d-2n};q^d)_{j_1+\dots+j_{d-2}}
(q^{d+n},q^{d-n};q^d)_{j_1+\dots+j_{d-1}}}& \notag\\[5pt]
\times\frac{(q^{r-(d-1)n};q^d)_{j_1+\dots+j_{d-1}}}
{(q^{d+r};q^d)_{j_1+\dots+j_{d-1}}}
\frac{q^{(d+r)(j_{d-2}+\dots+(d-2)j_1)} q^{d(j_1+\dots+j_{d-1})}}
{q^{2rj_1}\cdots q^{2r(j_1+\dots+j_{d-2})}}&.
\end{align*}
However, this sum vanishes due to the $m=d$, $q\mapsto q^d$, $a\mapsto q^r$,
$e_1\mapsto q^{d+(d-1)n}$, $e_i\mapsto q^{r+(d-i+1)n}$, $n_1=0$,
$n_i\mapsto (n+r-d)/d$, $2\leqslant i\leqslant d$, $N=((d-1)n-r)/d$, case
of Lemma~\ref{lem:ms=0}.
\end{proof}

\section{Proof of Theorem \ref{thm:2}}\label{sec:thm2}
We first give the following result, which is a generalization of \cite[Lemma 3.1]{GS}.
\begin{lemma}\label{lem:2.1}
Let $d$, $m$ and $n$ be positive integers with  $m\leqslant n-1$ and $dm\equiv -1\pmod{n}$.
Then, for $0\leqslant k\leqslant m$, we have
\begin{equation*}
\frac{(aq;q^d)_{m-k}}{(q^d/a;q^d)_{m-k}}
\equiv (-a)^{m-2k}\frac{(aq;q^d)_k}{(q^d/a;q^d)_k} q^{m(dm-d+2)/2+(d-1)k}
\pmod{\Phi_n(q)}.
\end{equation*}
\end{lemma}
\begin{proof}In view of $q^n\equiv 1\pmod{\Phi_n(q)}$, we have
\begin{align}\label{aqcong}
\frac{(aq;q^d)_{m} }{(q^d/a;q^d)_{m}}
&=\frac{(1-aq)(1-aq^{d+1})\cdots (1-aq^{dm-d+1})}
{(1-q^d/a)(1-q^{2d}/a)\cdots (1-q^{dm}/a)} \notag\\[5pt]
&\equiv \frac{(1-aq)(1-aq^{d+1})\cdots (1-aq^{dm-d+1})}
{(1-q^{d-dm-1}/a)(1-q^{2d-dm-1}/a)\cdots (1-q^{-1}/a)}\notag\\[5pt]
&=(-a)^{m}q^{m(dm-d+2)/2} \pmod{\Phi_n(q)}.
\end{align}
Furthermore, modulo $\Phi_n(q)$, we get
\begin{align*}
\frac{(aq;q^d)_{m-k}}{(q^d/a;q^d)_{m-k}}
&=\frac{(aq;q^d)_{m}}{(q^d/a;q^d)_{m}}
\frac{(1-q^{dm-dk+d}/a)(1-q^{dm-dk+2d}/a)\cdots (1-q^{dm}/a)}
{(1-aq^{dm-dk+1})(1-aq^{dm-dk+d+1})\cdots (1-aq^{dm-d+1})}
\\[5pt]
&\equiv \frac{(aq;q^d)_{m}}{(q^d/a;q^d)_{m}}
\frac{(1-q^{d-dk-1}/a)(1-q^{2d-dk-1}/a)\cdots (1-q^{-1}/a)}
{(1-aq^{-dk})(1-aq^{d-dk})\cdots (1-aq^{-d})} \\[5pt]
&=\frac{(aq;q^d)_{m}}{(q^d/a;q^d)_{m}} \frac{(aq;q^d)_k}{(q^d/a;q^d)_k} a^{-2k} q^{(d-1)k},
\end{align*}
which together with \eqref{aqcong} establishes the assertion.
\end{proof}

Similarly, we have the following $q$-congruence.
\begin{lemma}\label{lem:2.2}
Let $d$, $m$ and $n$ be positive integers with  $m\leqslant n-1$ and $dm\equiv 1\pmod{n}$.
Then, for $0\leqslant k\leqslant m$, we have
\begin{equation*}
\frac{(aq^{-1};q^d)_{m-k}}{(q^d/a;q^d)_{m-k}}
\equiv (-a)^{m-2k}\frac{(aq^{-1};q^d)_k}{(q^d/a;q^d)_k}
q^{m(dm-d-2)/2+(d+1)k}\pmod{\Phi_n(q)}.
\end{equation*}
\end{lemma}
The proof of Lemma~\ref{lem:2.2} is completely analogous to that of
Lemma~\ref{lem:2.1} and thus omitted.

\begin{proof}[Proof of Theorem \ref{thm:2}]
Since $\gcd(d,n)=1$, there exists a positive integer $m\leqslant n-1$ such that
$dm\equiv -1\pmod{n}$.
By the $a=1$ case of Lemma~\ref{lem:2.1} one sees that,
for $0\leqslant k\leqslant m$,
the $k$-th and $(m-k)$-th terms on the left-hand side of \eqref{eq:first-1}
cancel each other modulo $\Phi_n(q)$, i.e.,
\begin{align*}
[2d(m-k)+1]
\frac{(q;q^d)_{m-k}^{2d}}{(q^d;q^d)_{m-k}^{2d}}q^{d(d-2)(m-k)}
\equiv -[2dk+1]
\frac{(q;q^d)_k^{2d}}{(q^d;q^d)_k^{2d}}q^{d(d-2)k}
\pmod{\Phi_n(q)}.
\end{align*}
This proves that
\begin{equation}
\sum_{k=0}^{m}[2dk+1]
\frac{(q;q^d)_k^{2d}}{(q^d;q^d)_k^{2d}}q^{d(d-2)k} \equiv 0 \pmod{\Phi_n(q)}.  \label{eq:symmetry-1}
\end{equation}
Moreover, since $dm\equiv -1\pmod{n}$, the expression $(q;q^d)_k$ contains a factor of the form $1-q^{\alpha n}$ for $m< k\leqslant n-1$, and is therefore
congruent to $0$ modulo $\Phi_n(q)$. At the same time the expression $(q^d;q^d)_k$ is relatively prime to $\Phi_n(q)$ for $m< k\leqslant n-1$.
Therefore, each summand in \eqref{eq:first-1} with $k$ in the range $m< k\leqslant n-1$ is congruent to $0$ modulo $\Phi_n(q)$.
This together with \eqref{eq:symmetry-1} establishes the $q$-congruence \eqref{eq:first-1}.

Similarly, we can use Lemma~\ref{lem:2.2} to prove \eqref{eq:first-2}.
The proof of the theorem is complete.
\end{proof}

\section{Proof of Conjectures~\ref{conj-1} and \ref{conj-2}}\label{sec:conjectures}
As mentioned in the introduction, we only need to show that Conjectures~\ref{conj-1} and \ref{conj-2} are also true modulo $[n]$.
We first give a detailed proof of the $q$-congruences modulo $[n]$ in Conjecture~\ref{conj-1}.
\begin{proof}[Proof of Conjecture~\ref{conj-1}]
We need to show that
\begin{align}
\sum_{k=0}^{((d-1)n-1)/d}[2dk+1]
\frac{(q;q^d)_k^{2d}}{(q^d;q^d)_k^{2d}}q^{d(d-2)k} \equiv 0 \pmod{[n]},  \label{conj1-a}\\
\intertext{and}
\sum_{k=0}^{n-1}[2dk+1]
\frac{(q;q^d)_k^{2d}}{(q^d;q^d)_k^{2d}}q^{d(d-2)k} \equiv 0 \pmod{[n]}.  \label{conj1-b}
\end{align}

Let $\zeta\ne1$ be an $n$-th root of unity, not
necessarily primitive. Clearly, $\zeta$ is a primitive root of unity
of degree $s$ with $s\mid n$ and $s>1$.  Let $c_q(k)$ denote the $k$-th term on the
left-hand side of \eqref{conj1-a} or \eqref{conj1-b}, i.e.,
\begin{equation*}
c_q(k)=[2dk+1]
\frac{(q;q^d)_k^{2d}}{(q^d;q^d)_k^{2d}}q^{d(d-2)k}.
\end{equation*}
The $q$-congruences \eqref{eq:symmetry-1} and \eqref{eq:first-1} with
$n\mapsto s$ imply that
\begin{equation*}
\sum_{k=0}^{m}c_\zeta(k)=\sum_{k=0}^{s-1}c_\zeta(k)=0,
\end{equation*}
where $dm\equiv -1\pmod{s}$ and $1\leqslant m\leqslant s-1$.
Observing that
\begin{equation}
\frac{c_\zeta(\ell s+k)}{c_\zeta(\ell s)}
=\lim_{q\to\zeta}\frac{c_q(\ell s+k)}{c_q(\ell s)}
=c_\zeta(k),  \label{eq:conj1-aa}
\end{equation}
we have
\begin{equation}
\sum_{k=0}^{n-1}c_\zeta(k)=\sum_{\ell=0}^{n/s-1}
\sum_{k=0}^{s-1}c_\zeta(\ell s+k)
=\sum_{\ell=0}^{n/s-1}c_\zeta(\ell s) \sum_{k=0}^{s-1}c_\zeta(k)=0,  \label{eq:conj1-bb}
\end{equation}
and
\begin{equation*}
\sum_{k=0}^{((d-1)n-1)/d}c_\zeta(k)
=\sum_{\ell=0}^{N-1} c_\zeta(\ell s)
\sum_{k=0}^{s-1}c_\zeta(k)+c_\zeta(N s)\sum_{k=0}^{m}c_\zeta(k)=0,
\end{equation*}
where
$$
N=\frac{(d-1)n-dm-1}{ds}.
$$
(It is easy to check that $N$ is a positive integer.)
This means that the sums $\sum_{k=0}^{n-1}c_q(k)$ and $\sum_{k=0}^{((d-1)n-1)/d}c_q(k)$
are both divisible by the cyclotomic polynomial $\Phi_s(q)$.
Since this is true for any divisor $s>1$ of $n$, we deduce that they
are divisible by
\begin{equation*}
\prod_{s\mid n,\, s>1}\Phi_s(q)=[n],
\end{equation*}
thus establishing the $q$-congruences \eqref{conj1-a} and \eqref{conj1-b}.
\end{proof}

Similarly, we can prove Conjecture~\ref{conj-2}.
\begin{proof}[Proof of Conjecture~\ref{conj-2}]
This time we need to show that
\begin{align}
\sum_{k=0}^{((d-1)n+1)/d}
[2dk-1]\frac{(q^{-1};q^d)_k^{2d}}{(q^d;q^d)_k^{2d}}q^{d^2 k}  \equiv 0 \pmod{[n]},  \label{conj2-a}\\
\intertext{and}
\sum_{k=0}^{n-1}
[2dk-1]\frac{(q^{-1};q^d)_k^{2d}}{(q^d;q^d)_k^{2d}}q^{d^2 k}  \equiv 0 \pmod{[n]}.  \label{conj2-b}
\end{align}
Again, let $\zeta$ be a primitive root of unity
of degree $s$ with $s\mid n$ and $s>1$, and let
\begin{equation*}
c_q(k)=
[2dk-1]\frac{(q^{-1};q^d)_k^{2d}}{(q^d;q^d)_k^{2d}}q^{d^2 k}.
\end{equation*}
Just like before, we have
\begin{equation*}
\sum_{k=0}^{m}c_\zeta(k)=\sum_{k=0}^{s-1}c_\zeta(k)=0,
\end{equation*}
where $dm\equiv 1\pmod{s}$ and $1\leqslant m\leqslant s-1$.
Furthermore, we also have \eqref{eq:conj1-aa}, \eqref{eq:conj1-bb}, and
\begin{equation*}
\sum_{k=0}^{((d-1)n+1)/d}c_\zeta(k)
=\sum_{\ell=0}^{N-1} c_\zeta(\ell s)
\sum_{k=0}^{s-1}c_\zeta(k)+c_\zeta(N s)\sum_{k=0}^{m}c_\zeta(k)=0,
\end{equation*}
where
$
N=\frac{(d-1)n-dm+1}{ds}
$
this time. The rest is exactly the same as in the proof of Conjecture~\ref{conj-1} and
is omitted here.
\end{proof}

\section{An open problem}\label{sec:final}

Recently, the first author \cite[Theorem 5.4]{Guo-m3} proved that
\begin{align}
\sum_{k=0}^{M}[4k-1]_{q^2}[4k-1]^2\frac{(q^{-2};q^4)_k^4}{(q^4;q^4)_k^4}q^{4k}
\equiv 0 \pmod{[n]_{q^2}\Phi_n(q^2)^2},\label{last-1}
\end{align}
where $n$ is odd and $M=(n+1)/2$ or $n-1$. We take this opportunity to
propose a unified generalization of \cite[Conjectures 6.3 and 6.4]{Guo-m3},
involving a remarkable $q$-hypergeometric congruence modulo the fifth
power of a cyclotomic polynomial:
\begin{conjecture}
Let $n>1$ be an odd integer. Then
\begin{align*}
\sum_{k=0}^{M}[4k-1]_{q^2}[4k-1]^2\frac{(q^{-2};q^4)_k^4}{(q^4;q^4)_k^4}q^{4k}
\equiv (2q+2q^{-1}-1)[n]_{q^2}^4 \pmod{[n]_{q^2}^4\Phi_n(q^2)},
\end{align*}
where $M=(n+1)/2$ or $n-1$.
\end{conjecture}

\end{document}